\documentclass{amsart}
\usepackage{amssymb,latexsym, amsmath, amscd, array, graphicx}

\swapnumbers
\numberwithin{equation}{section}

\theoremstyle{plain}

\newtheorem{thm}{Theorem}[section]
\newtheorem{cor}[thm]{Corollary}

\newtheorem{prop}[thm]{Proposition}

\theoremstyle{definition}

\newtheorem{remark[thm]}{Remark}

\def\cat{\protect\operatorname{cat}}

\def\Z{{\mathbb Z}}

\def\R{{\mathbb R}}

\def\bA{{\mathbb A}}

\def\1{\hbox{\rm\rlap {1}\hskip.03in{\rom I}}}
\def\Bbbone{{\rm1\mathchoice{\kern-0.25em}{\kern-0.25em}
{\kern-0.2em}{\kern-0.2em}I}}

\long\def\forget#1\forgotten{} %

\begin{document}

\title[On homotopy cofiber ]
{The topological complexity and the homotopy cofiber of the diagonal map  for non-orientable surfaces}
\author[A.~Dranishnikov]
{Alexander Dranishnikov}
\address{A. Dranishnikov, Department of Mathematics, University
of Florida, 358 Little Hall, Gainesville, FL 32611-8105, USA}
\email{dranish@math.ufl.edu}
\thanks{Supported by NSF, grant DMS-1304627}

\begin{abstract} We show that the Lusternik-Schnirelmann category of the homotopy cofiber of the diagonal map
of non-orientable surfaces equals three.

Also, we prove that the topological complexity of non-orientable surfaces of genus $\ge 4$ is four. 
\end{abstract}

\maketitle

\section{Introduction}

 The {\em topological complexity}  $TC(X)$ of a space $X$ was defined by Farber~\cite{F} as an invariant that measures
the navigation complexity of $X$ regarded as the configuration space for a robot motion planning. By a slightly modified definition $TC(X)$   is
the minimal number $k$ 
such that $X\times X$ admits a cover by $k+1$ open sets  $U_0,\dots, U_k$ such that over each $U_i$ there is 
a continuous  motion planning algorithm $s_i:U_i\to PX$, i.e. a continuous map of $U_i$ to the path space  $PX=X^{[0,1]}$ with
$s_i(x,y)(0)=x$ and $s_i(x,y)(1)=y$ for all $(x,y)\in U_i$. Here we have defined the reduced topological complexity. 
The original (nonreduced) topological complexity is by one larger.

The topological complexity of an orientable surface of genus $g$ was computed in~\cite{F}: $$TC(\Sigma_g)=\begin{cases}2, &\text{if}\ g=0,1\\
                                                              4 &  \text{if}\ g>1.\end{cases}$$

For the non-orientable surfaces of genus $g>1$ the complete answer is still unknown. What was known are the bounds:
$3\le TC(N_g)\le 4$ and the equality $TC(\R P^2)=3$. In this paper we show that  $TC(N_g)=4$ for $g\ge 4$. The cases of $g=2$ and $g=3$ are still open.

The topological complexity is a numeric invariant of topological  spaces similar to the Lusternik-Schnirelmann category.
It is unclear if $TC$ can be completely reduced to the LS-category. One attempt of such reduction (\cite{GV2},~\cite{Dr2}) deals with the 
problem whether the topological complexity $TC(X)$ coincides with the Lusternik-Schnirelmann category
$\cat (C_{\Delta X})$ of the homotopy cofiber of the diagonal map $\Delta: X\to X\times X$,
$C_{\Delta X}=(X\times X)/\Delta X$. The coincidence of 
these two concepts was proven in~\cite{GV2} for a large class of spaces. Also
in~\cite{GV1} the equality was proven for the weak 
in the sense of Berstein and Hilton versions of $TC$ and $\cat$.

In this paper we prove that $\cat(C_{\Delta N})=3$ for any non-orientable surface $N$. Thus, in view of the computation
$TC(N_g)=4$ for $g>4$ we obtain counterexamples to the conjecture  $TC(X)=\cat(C_{\Delta X})$.

Since both computations are rather technical, at the end of the paper we present a short counterexample: Higman's group.
We show that $TC(BH)\ne\cat(C_{\Delta BH})$  where $BH=K(H,1)$ is the classifying space for Higman's group $H$. 
The proof of that is short because the main difficulty there, the proof of the equality $TC(BH)=4$, is hidden behind the reference~\cite{GLO}.

The author is thankful to Mark Grant and Jesus Gonzalez for pointing out on mistakes in early versions of this paper.

\section{Preliminaries}

\subsection{Category of spaces}
By the definition the Lusternik-Schnirelmann category $\cat X\le k$ for a topological space $X$ if there is 
a cover $X=U_0\cup\dots\cup U_k$ by $k+1$ open subsets each of which is contractible in $X$.

Let $\pi=\pi_1(X)$. We recall that the cup product $\alpha\smile\beta$ of twisted cohomology classes $\alpha\in H^i(X;L)$ and $\beta\in H^j(X;M)$ takes value in $H^{i+j}(X;L\otimes M)$ where $L$ and $M$ are $\pi$-modules and $L\otimes M$ is the tensor product over $\Z$~\cite{Bro}. Then the cup-length of $X$, denoted as $c.l.(X)$, is defined as the maximal integer $k$ such that $\alpha_1\smile\dots\smile\alpha_k\ne 0$ for some $\alpha_i\in H^{n_i}(X;L_i)$ with $n_i>0$.
The following inequalities give  estimates on the LS-category~\cite{CLOT}:
\begin{thm}\label{cl}
$c.l.(X)\le\cat X\le\dim X.$

If $X$ is $k$-connected, then $\cat X\le\dim X/(k+1)$.
\end{thm}

\subsection{Category of maps}
We recall that the LS-category of a map $f:Y\to X$ is the least integer $k$ such that $Y$ can be covered by
$k+1$ open sets $U_0,\dots, U_k$ such that the restrictions $f|_{U_i}$ are null-homotopic for all $i$.

The following two facts are proven in~\cite{Dr3} (Proposition 4.3 and Theorem 4.4):
\begin{thm}\label{deform}
Let $u:X\to B\pi$ be a map classifying the universal covering of a CW complex $X$. Then the following are equivalent:

(1) $\cat(u)\le k$; 

(2)  $u$ is homotopic to a map $f:X\to B\pi$ with $f(X)\subset B\pi^{(k)}$.
\end{thm}

\begin{thm}\label{space-map} 
Let $X$ be an $n$-dimensional CW complex whose universal covering $\tilde X$  is $(n-k)$-connected. Suppose that $X$ admits a classifying map $u:X\to B\pi$ with $\cat u\le k$.  Then
$
\cat X\le k.
$
\end{thm}

\subsection{Inessential complexes}
One can extend Gromov's theory of inessential manifolds~\cite{Gr} to simplicial complexes and, in particular, to 
pseudo-manifolds. We call an $n$-dimensional complex $X$ {\em inessential} if
a map $u:X\to B\pi$ that classifies the universal covering of $X$ can be deformed to the $(n-1)$-dimensional skeleton.
Otherwise it is called {\em essential}.

\begin{prop}\label{KR}
An $n$-dimensional complex $X$ is inessential if and only if $\cat X\le n-1$.
\end{prop}
\begin{proof}
Suppose that $\cat X\le n-1$. Then $\cat(u)\le n-1$ where $u:X\to B\pi$ is a classifying map. By Theorem~\ref{deform}, $X$ is inessential.

If $X$ is inessential, by Theorem~\ref{deform} $\cat(u)\le n-1$. We apply Theorem~\ref{space-map} to $X$ with $k=n-1$ to obtain that
$\cat X\le n-1$.
\end{proof}
REMARK. Proposition~\ref{KR} in the case when $X$ is a closed manifold was proven in~\cite{KR}.

\subsection{Pseudo-manifolds}
We recall that an  $n$-dimensional {\em pseudo-manifold} is a simplicial complex $X$ which is pure, nonbranching and strongly connected.
{\em Pure} means that $X$ is the union of $n$-simplices. {\em Nonbranching} means that there is a subpolyhedron $S\subset X$ of dimension $\le n-2$ such that
$X\setminus S$ is an $n$-manifold. {\em Strongly connected} means that every pair of $n$-simplices $\sigma,\sigma'$ in $X$ can be joined by a chain of simplices
$\sigma_0,\dots,\sigma_m$ with $\sigma_0=\sigma$, $\sigma_m=\sigma'$, and $\dim(\sigma_i\cap\sigma_{i-1})=n-1$ for $i=1,\dots,m$.
Note that every $n$-dimensional pseudo-manifold $X$ admits a CW complex structure with one vertex and one $n$-dimensional cell.

A sheaf $\mathcal O_X$ on an $n$-dimensional pseudo-manifold $X$ generated by the presheaf $U\to H_n(X,X\setminus U)$ is called the {\em orientation
sheaf}. We recall that in case of manifolds the orientation sheaf $\mathcal O_{N}$ on $N$ is defined as the pull-back of 
the canonical $\Z$-bundle $\mathcal O$ on $\R P^{\infty}$ by the map $w_1:N\to\R P^{\infty}$
that represents the first Stiefel-Whitney class. 

A pseudo-manifold $X$ is l{\em ocally orientable} if $\mathcal O_X$ is locally constant with the stalks isomorphic to $\Z$.
For a locally orientable $n$-dimensional pseudo-manifold $X$, $H_n(X;\mathcal O_X)=\Z$, and the $n$-dimensional cell (we may assume that it is unique) defines 
a generator of $\Z$ called the fundamental class $[X]$ of $X$.
\begin{thm}
Let $X$ be a locally orientable $n$-dimensional pseudo-manifold and let $A$ be a $\pi_1(X)$-module. Then the cap product with $[X]$ defines an isomorphism
$$
[X]\cap:H^n(X;A)\to H_0(X;A\otimes\underline{Z})
$$
where $\underline Z$ stands for the $\pi_1(X)$-module $\Z$ defined by the orientation sheaf $\mathcal O_X$.
\end{thm}
\begin{proof} We note that in these dimensions the proof of the classical Poincare Duality for locally oriented manifolds (\cite{Bre}) works for pseudo-manifolds as well.
\end{proof}

\begin{prop}\label{pseudo}
Suppose that a map $f:M\to B\pi$ of a closed $n$-dimensional locally orientable pseudo-manifold induces an epimorphism of the fundamental groups. 
Suppose that the orientation sheaf on $M$ is the image under $f^*$ of a sheaf on $B\pi$. Then $f$ can be deformed to the
$(n-1)$-skeleton $B\pi^{(n-1)}$ if and only if $f_*([M])=0$ where $[M]$ is the fundamental class.
\end{prop}
\begin{proof}
The 'only if' direction follows from the dimensional reason and the fact that $f_*$ does not change under a homotopy. 

Let $f_*([M])=0$. We show that the primary obstruction $o_f$ for deformation of $f$ to the $(n-1)$-skeleton is trivial. Since $o_f$
is the image of the primary obstruction $o'$ to deformation of $B\pi$ to $B\pi^{(n-1)}$, it suffices to prove the equality
$f^*(o')=0$. Note that 
$$
f_*([M]\cap f^*(o'))=f_*([M])\cap o'=0.
$$
Since $f$ induces an epimorphism of the fundamental groups, it induces isomorphism of 0-dimensional homology groups with any local coefficients.
Hence, $[M]\cap f^*(o')=0$. Since in dimension 0 the Poincare Duality holds  for locally orientable pseudo-manifolds, we obtain that $f^*(o')=0$.
\end{proof}

\subsection{Homology of projective space}
We denote by $\mathcal O$ the canonical local coefficient system on  the projective space $\R P^{\infty}$ with the fiber $\Z$.

\begin{prop}\label{even-odd}
$$H_i(\R P^{\infty};\mathcal O)=\begin{cases}\Z_2, &\text{if}\  i\ \text{is even}\\ 
                                                                      0 &  \text{if}\  i\ \text{is odd} .\end{cases}$$
\end{prop}
\begin{proof} Let $\underline Z$ denote a $\Z_2$-module $\Z$ with the flip over 0 action.
We note that $H_i(\R P^{\infty};\mathcal O)=H_i(\Z_2,\underline Z)$. If $\Z_2=\{1,t\}$, then the homology groups $H_*(\Z_2,\underline Z)$
are the homology of the chain complex (\cite{Bro})
$$
\begin{CD}
\cdots @>1-t>> \underline Z @>1+t>>\underline Z @>1-t>>\underline Z @>1+t>>\underline Z @>1-t>>\underline Z
\end{CD}
$$
which is the complex
$$
\begin{CD}
\cdots @>2>>\Z @>0>>\Z @>2>>\Z @>0>>\Z @>2>>\Z.
\end{CD}
$$
\end{proof}
\subsection{Schwarz genus}
We recall that the {\em Schwarz genus} $g(f)$ of a fibration $f:E\to B$ is the minimal number $k$ such that $B$ can be covered by $k$ open sets
on which $f$ admits a section~\cite{Sch}. Then $\cat X+1=g(c:\ast\to X)$ and $TC(X)+1=g(\Delta:X\to X\times X)$ where the constant map $c$
and the diagonal map $\Delta$ are assumed to be represented by fibrations. Schwarz connected the genus $g(f)$ with the existence of a section of a special
fibration constructed out of $f$ by means of an operation that generalizes the Whitney sum. 

Here is the construction:
Given two maps~$f_1:X_1\to Y$ and~$f_2:X_2\to Y$, we define the {\em fiberwise join} of spaces $X_1$ and $X_2$ as
\[
X_1\ast_YX_2=\{tx_1+(1-t)x_2\in X_1\ast X_2\mid f_1(x_1)=f_2(x_2)\}
\]
and define the {\em fiberwise join} of~$f_1$, $f_2$
as the map
\[
f_1{\ast_Y}f_2:X_1\ast_YX_2\to Y,\quad \text{with}\ \ (f_1{\ast_Y}f_2)(tx_1+(1-t)x_2)=f_1(x_1)=f_2(x_2).
\]
The iterated $n$ times fiberwise join product of a map $f:E\to B$ is denoted as $\ast^n_B:\ast^n_BE\to B$.
\begin{thm}[\cite{Sch}, Theorem 3]\label{Schwarz}
For a fibration $f:E\to B$ the Schwarz genus $g(f)\le n$ if and only if $\ast_b^nf:\ast_B^n\to E$ admits a section.
\end{thm}
Alos Schwarz proved the following~\cite{Sch}:
\begin{thm}\label{Schwarz2}
Let $\beta$ be the primary obstruction to a section of a fibration $f:E\to B$.
Then $\beta^n$ is the primary obstruction to a section of $\ast^n_bf$.
\end{thm}

Let~$p^X:PX\to X\times X$ be the end points map: $p^X(\phi)=(\phi(0),\phi(1))\in X\times X$.  Here
$PX$ is the space of all paths $\phi:[0,1]\to X$ in~$X$.  Clearly, $p^X$ is a Serre path fibration
that represents the diagonal map $\Delta:X\to X\times X$. Let $P_0X\subset PX$ be the space of paths
issued from a base point $x_0$ and let $\tilde p^X:P_0X\to X$ is define as $\tilde p^X(\phi)=\phi(1)$.
Then $\tilde p^X$ is a fibration representative for the map $x_0\to X$. 

For $n>0$ we denote by $p_n^X=\ast^{n+1}_{X\times X}p^X$ and $\Delta_n(X)=\ast^{n+1}_{X\times X}PX$.
Thus, 
elements of $\Delta_n(X)$ can be viewed as formal linear combinations
$\sum_{i=0}^nt_i\phi_i$ where $\phi_i:[0,1]\to X$ with $\phi_1(0)=\cdots=\phi_n(0)$,
$\phi_1(1)=\cdots=\phi_n(1)$, $t_i\ge 0$, and $\sum t_i=1$.

Similarly we use notations  $\tilde p^{X}_n=\ast^{n+1}_X\tilde p^X$ and $G_n(X)=\ast^{n+1}_{X}P_0X$.
The fibration  $\tilde p^{X}_n$ is called the $n$-th Ganea fibration.
Theorem~\ref{Schwarz} stated for these fibrations
produces the following
\begin{thm}\label{ganea}
For a~CW-space~$X$,~

(1) $TC(X)\le n$ if and only if there exists a
section of~$p_n^X:\Delta_n(X)\to X\times X$;

(2) $\cat X\le n$ if and only if there exists a
section of~$\tilde p_n^X:G_n(X)\to X$.
\end{thm}

We note that the fiber of both fibrations $p^X$ and $\tilde p^X$ is the loop space $\Omega X$.
The fiber $F_n=(p^X_n)^{-1}(x_0)$ of the fibration $p^X_n$ is the join product $\Omega X\ast\dots\ast\Omega X$ of $n+1$ copies of the loop space
$\Omega X$ on $X$. So, $F_n$ is $(n-1)$-connected.

A continuous map $f:X\to Y$ for any $n$ defines the commutative diagram
$$
\begin{CD}
\Delta_n(X) @>>>\Delta_n(Y)\\
@V p_n^XVV @Vp_n^YVV\\
X\times X @>f\times f>> Y\times Y.\\
\end{CD}
$$
\begin{cor}\label{lift}
If $TC(X)\le n$, then for any $f:X\to Y$ the map $f\times f$ admits a lift with respect to $p_n^Y$.
\end{cor}
\subsection{Berstein-Schwarz cohomology class}
The Berstein-Schwarz class of a discrete group $\pi$ is the first obstruction $\beta_{\pi}$ to a lift of $B\pi=K(\pi,1)$ to the universal covering $E\pi$. 
Note that $\beta_{\pi}\in H^1(\pi,I(\pi))$ where $I(\pi)$ is the augmentation ideal of the group ring $\Z\pi$~\cite{Be},\cite{Sch}.

\begin{thm}[Universality~\cite{DR},\cite{Sch}]\label{universal}
For any cohomology class $\alpha\in H^k(\pi,L)$ there is a homomorphism of $\pi$-modules $I(\pi)^k\to L$ such that the induced homomorphism for cohomology takes $(\beta_{\pi})^k\in H^k(\pi;I(\pi)^k)$ to $\alpha$  where $I(\pi)^k=I(\pi)\otimes\dots\otimes I(\pi)$ and $(\beta_{\pi})^k=\beta_{\pi}\smile\dots\smile\beta_{\pi}$.
\end{thm}

\section{Computation of the LS-category of the cofiber}

For $X=\R P^n$, $n>1$, the equality $TC(X)=\cat(C_{\Delta X})$ was established in \cite{GV2}.
Together with the computation $TC(\R P^2)=3$ from~\cite{FTY} it gives the following
\begin{thm}\label{FTY}
$$\cat((\R P^2\times\R P^2)/\Delta\R P^2)=3.$$
\end{thm}

\subsection{Free abelian topological groups} Let $\bA(N)$ denote the free abelian topological group generated by $N$ (see~\cite{M},\cite{G} or \cite{Dr1}). Let $j:N\to \bA(N)$ be the natural inclusion.
By the Dold-Thom theorem~\cite{DT} (see also~\cite{Dr1}), $\pi_i(\bA(N))=H_i(N)$ and $j_*:\pi_i(N)\to \pi_i(\bA(N))$ is the Hurewicz homomorphism.
Therefore, $\bA(\R P^2)$ is homotopy equivalent to $\R P^{\infty}$. Moreover, for a non-orientable surface $N$ of genus $g$ the space $\bA(N)$
is homotopy equivalent to $\R P^{\infty}\times T^{g-1}$ where $T^m=S^1\times\cdots\times S^1$ denotes the $m$-dimensional torus. 

Let $\tilde{\mathcal O}$ be the twisted coefficient system on $\bA(N)$ that comes from the canonical 
system $\mathcal O$ on $\R P^{\infty}$ as the pull-back under the projection $\R P^{\infty}\times T^{g-1}\to\R P^{\infty}$.

\begin{prop}\label{tor}
For any non-orientable surface $N$ 
$$H_2(\bA(N);\tilde{\mathcal O})=\oplus\Z_2.$$
\end{prop}
\begin{proof}
 By the Kunneth formula for local coefficients~\cite{Bre}, $$
 (*)\ \ \  H_2(\bA(N);\tilde{\mathcal O})=H_0(T^{g-1})\otimes H_2(\R P^{\infty};\mathcal O)$$
$$\oplus H_1(T^{g-1})\otimes H_1(\R P^{\infty};\mathcal O)$$
$$\oplus H_2(T^{g-1})\otimes H_0(\R P^{\infty};\mathcal O).$$ The Tor part of the Kunneth formula is trivial since the second factors has torsion free homology groups.
Thus, taking into account Proposition~\ref{even-odd} we obtain $$H_2(A(N);\tilde{\mathcal O})=H_2(\R P^{\infty};\mathcal O)\oplus (\Z_2\otimes H_2(T^{g-1}))=
\Z_2\oplus H_2(T^{g-1};\Z_2)=\oplus\Z_2.$$
\end{proof}

For a topological abelian group $A$ we denote by $\mu_A=\mu:A\times A\to A$ 
the continuous group homomorphism defined by the formula $\mu(a,b)=a-b$. 
\begin{prop}\label{pull-back0}
Let $N=\R P^2$. Then the pull-back $(j\times j)^*\mu^*(\mathcal O)$ is the $\Z$-orientation sheaf for the manifold $\R P^2\times \R P^2$
where $\mu=\mu_{\bA(\R P^2)}$ and $\mathcal O$ is the canonical $\Z$-bundle over $\bA(\R P^2)=\R P^{\infty}$. 
\end{prop}
We make a forward reference to Proposition~\ref{pull-back} for the proof.

\begin{prop}\label{a}
Let $a\in H_2(\R P^\infty;\mathcal O)$ be a generator. Then $\mu_*(a\otimes a)=0$ where $\mu=\mu_{A(\R P^2)}$.
\end{prop}
\begin{proof}
Note that $\pi_1((\R P^2\times\R P^2)/\Delta\R P^2)=\Z_2$
 (see Proposition~\ref{abelinization}). Let $$f:(\R P^2\times\R P^2)/\Delta\R P^2)\to \bA(\R P^2)$$ be a map that induces an isomorphism of the fundamental groups. We claim that the map $\mu\circ(j\times j)$ is homotopic to $f\circ q$ where $q:\R P^2\times\R P^2\to
(\R P^2\times\R P^2)/\Delta\R P^2$ is the quotient map. This holds true since both maps induce the same homomorphism of the fundamental groups.
In view of Theorem~\ref{FTY} and Proposition~\ref{KR} the map $f$ is homotopic to a map with the image in  the 3-dimensional skeleton. Therefore,
$f_*q_*(b\otimes b)=0$ where $b$ is a generator of $H_2(\R P^2;\mathcal O_{\R P^2})=\Z$. Note that $j_*(b)=a$. Then $\mu_*(a\otimes a)=\mu_*(j\times j)_*(b\otimes b)=f_*q_*(b\otimes b)=0$.
\end{proof}
Let $N=N_g=\#_g\R P^2$ be a non-orientable surface of genus $g$.
Let $\pi=\pi_1(N)$ and $G=Ab(\pi)=H_1(N)=\Z_2\oplus\Z^{g-1}$. We recall that by the Dold-Thom theorem~\cite{DT},\cite{Dr1}
the space $\bA(N)$ is homotopy equivalent to $K(G,1)\sim \R P^\infty\times T^{g-1}$. 
\begin{prop}\label{h-h-e}
There is a homomorphism of topological abelian groups $$h:\bA(N_g)\to \bA(\R P^2)\times T^{g-1}$$ which is a homotopy equivalence.
\end{prop}
\begin{proof}
Since the spaces are $K(G,1)$s, it suffices to construct a homomorphism that induces an isomorphism of the fundamental groups.
For that it suffices to construct a map $f:N_g\to \bA(\R P^2)\times T^{g-1}$ that induces an isomorphism of integral 1-dimensional
homology groups.
Then the homomorphism $h$ is an extension to $\bA(N_g)$ which exists by the universal property of  free abelian topological groups.
We consider two cases.

(1) If $g$ is odd, then $N_g=M_{(g-1)/2}\#\R P^2$. We define $f$ as the composition
$$
M_{(g-1)/2}\#\R P^2\stackrel{q}\to M_{(g-1)/2}\vee\R P^2\stackrel{\phi\vee j}\rightarrow T^{g-1}\vee \bA(\R P^2)\stackrel{i}\to T^{g-1}\times \bA(\R P^2)$$
where $q$ is  collapsing of the separating circle in the connected sum, $\phi$ is a map that induces isomorphism of 1-dimensional homology, and $i$ is the inclusion.
It is easy to check that $f$ induces an isomorphism $f_*:H_1(N_g)\to H_1(\bA(\R P^2)\times T^{g-1})$.

(2) If $g$ is even, then $N_g=M_{(g-2)/2}\#K$ where $K$ is the Klein bottle. There is a homotopy equivalence $s:\bA(K)\to S^1\times \bA(\R P^2)$.
We define $f$ as the composition
$$
M_{(g-2)/2}\#K\stackrel{q}\to M_{(g-1)/2}\vee K\stackrel{\phi\vee s\circ j}\rightarrow T^{g-2}\vee (S^1\times\bA(\R P^2)) \stackrel{i}\to T^{g-2}\times S^1\times \bA(\R P^2)$$
where $q$ is  collapsing of the connecting circle, $\phi$ is a map that induces isomorphism of 1-dimensional homology, and $i$ is the inclusion.
One can check that $f$ induces an isomorphism $f_*:H_1(N_g)\to H_1(\bA(\R P^2)\times T^{g-1})$.
\end{proof}

\begin{prop}\label{product mu}
For abelian topological groups $A$ and $B$, $\mu_{A\times B}=\mu_A\times\mu_B$.
\end{prop}
\begin{proof} For all $a,a'\in A$ and $b,b'\in B$ we have
$$\mu_{A\times B}(a\times b, a'\times b')=(a-a')\times (b-b')=\mu_A(a)\times\mu_B(b)=(\mu_A\times\mu_B)(a\times b).$$
\end{proof}

\subsection{Twisted fundamental class}
The pull-back of the canonical $\Z$-bundle $\mathcal O$ over $\R P^{\infty}$ under the  projection $\R P^\infty\times T^{g-1}\to \R P^\infty$ defines a local coefficient system $\tilde{\mathcal O}$ on $\bA(N)$ with the fiber $\Z$.
On the $G$-module level the action of the fundamental group on $\Z$ is generated by the projection homomorphism $p:\Z_2\oplus\Z^{g-1}\to\Z_2$.
We note that $\mathcal O_N=j^*(\tilde{\mathcal O})$.

For sheafs $\mathcal A$ and $\mathcal B$ on $X$ and $Y$ we use the notation $\mathcal A\hat\otimes\mathcal B$
for $pr_X^*\mathcal A\otimes pr_Y^*\mathcal B$ where $pr_X:X\times Y\to X$ and $pr_Y:X\times Y\to Y$ are the projections. 
We note that if $\mathcal O_X$ and $\mathcal O_Y$ are the orientation sheafs on manifolds $X$ and $Y$, then
$\mathcal O_X\hat\otimes\mathcal O_Y$ is the orientation sheaf on $X\times Y$.

\begin{prop}\label{cross}
The cross product $[N]_{\mathcal O_N}\times[N]_{\mathcal O_N}$ is a fundamental class for $N\times N$.
\end{prop}
\begin{proof}
Note that for the orientation sheaf ${\mathcal O}_{N\times N}$ on the manifold $N\times N$ we have
$H_4(N\times N;{\mathcal O}_{N\times N})=\Z$. The Kunneth formula implies 
 $$ \Z=H_4(N\times N;{\mathcal O}_{N\times N})=H_4(N\times N;\mathcal O_{N}\hat{\otimes}\mathcal O_{N})=H_2(N;\mathcal O_N)\otimes H_2(N;\mathcal O_N)=\Z\otimes\Z.$$ Thus $[N_{\mathcal O_N}]\otimes[N_{\mathcal O_N}]$ is a
generator in $H_4(N\times N;{\mathcal O}_{N\times N})$.
\end{proof}

\begin{prop}
For $\mu=\mu_{\bA(N)}$,
$$\mu^*\tilde{\mathcal O}=\tilde{\mathcal O}\hat\otimes\tilde{\mathcal O}.$$
\end{prop}
\begin{proof}
Let $p:\Z_2\oplus\Z^{g-1}\to\Z_2$ be the projection.
The sheaf on left is defined by the representation $p\mu_*:\pi_1(\bA(N)\times\bA(N))\to Aut(\Z)=\Z_2$. The sheaf on the right is defined by
the representation $$\alpha(p\times p):\pi_1(\bA(N))\times\pi_1(\bA(N))\to\Z_2$$ where $\alpha:\Z_2\times\Z_2\to \Z_2$,
$\alpha(x,y)=x+y$, is the addition homomorphism. 
It is easy to check that these two coincide on the generating set $$(\pi_1(\bA(N))\times 1)\cup(1\times\pi_1(\bA(N)\subset\pi_1(\bA(N)\times\bA(N)).$$
Hence, they coincide on $\pi_1(\bA(N)\times\bA(N))$. Therefore, these sheafs are equal.
\end{proof}

\begin{cor}\label{pull-back}
The pull-back $(j\times j)^*\mu^*(\tilde{\mathcal O})$ is the orientation sheaf for the manifold $N\times N$. 
\end{cor}
\begin{proof}
Since $\mathcal O_N=j^*\tilde{\mathcal O}$, in view of Proposition~\ref{pull-back}, 
$$\mathcal O_{N\times N}=\mathcal O_n\hat\otimes\mathcal O_N=j^*\tilde{\mathcal O}\hat\otimes j^*\tilde{\mathcal O}=(j\times j)^*(\tilde{\mathcal O}\hat\otimes\tilde{\mathcal O})=(j\times j)^*\mu^*(\tilde{\mathcal O}).$$
\end{proof}

\begin{prop}\label{inverse}
Let $I:\bA(N)\to \bA(N)$, $I(x)=-x$, be the taking the inverse map. Then $I$ fixes every local system $\mathcal M$ on $\bA(N)$ and defines the identity homomorphism in homology $I_*:H_*(\bA(N);\mathcal M)\to H_*(\bA(N);\mathcal M)$.
\end{prop}
\begin{proof}
In view of Proposition~\ref{h-h-e} it suffices to prove it for $\bA(\R P^2)\times T^k$. Note that the inverse homomorphism $I:\bA(\R P^2)\times T^k\to \bA(\R P^2)\times T^k$ is the product of the inverse homomorphisms $I^1$ and $I^2$ for $\bA(\R P^2)$ and $T^k$ respectively. 
Also note that both $I^1:\bA(\R P^2)\to \bA(\R P^2)$ and $I^2:T^k\to T^k$ are  homotopic to the identity.
Thus, $I$ defines the identity automorphism of the fundamental group and, hence,  fixes $\mathcal M$. Then the homomorphism $I_*:H_*(\bA(N);\mathcal M)\to H_*(\bA(N);\mathcal M)$ is defined and $I_*=1$.
\end{proof}

\begin{prop}\label{main}
For a non-orientable surface $N$ the homomorphism
$$
(\mu_{\bA(N)})_*(j\times j)_*:H_4(N\times N;\mathcal O_{N\times N})\to H_4(\bA(N);\tilde{\mathcal O})
$$
is well-defined and
$$(\mu_{\bA(N)})_*(j\times j)_*([N\times N]_{{\mathcal O}_{N\times N}})= 0$$
where $[N\times N]_{{\mathcal O}_{N\times N}}\in H_4(N\times N;{\mathcal O}_{N\times N})$ is the fundamental class.
\end{prop}
\begin{proof}
By Proposition~\ref{pull-back}, $\mathcal O_{N\times N}=(j\times j)^*\mu^*(\tilde{\mathcal O})$ and, hence, the homomorphism
$$
\mu_*(j\times j)_*:H_4(N\times N;\mathcal O_{N\times N})\to H_4(\bA(N);\tilde{\mathcal O})
$$
is well-defined.

As before, we replace $\bA(N)$ by $\bA(\R P^2)\times T^{g-1}$.

By Proposition~\ref{cross} the cross product $$[N]_{\mathcal O_N}\times[N]_{\mathcal O_N}\in H_4(N\times N;\mathcal O_N\hat\otimes\mathcal O_N)$$ is
a fundamental class:  $[N\times N]_{{\mathcal O}_{N\times N}}=\pm [N]_{\mathcal O_N}\times[N]_{\mathcal O_N}$. 

Since $\mathcal O_N=j^*\tilde{\mathcal O}$, the homomorphism $j_*:H_2(N;\mathcal O_N)\to H_2(\bA(N);\tilde{\mathcal O})$ is well-defined.
In view of the Kunneth formula (see (*) in the proof of Proposition~\ref{tor}) we obtain
$$j_*([N]_{\mathcal O_N})=a\otimes 1_B+1_A\otimes b\in (H_2(A(\R P^2);\mathcal O_N)\otimes\Z)\oplus (\Z_2\otimes H_2(T^{g-1}))= H_2(A(N);\tilde{\mathcal O})$$
with $a\in H_2(\bA(\R P^2);\mathcal O)=H_2(\R P^{\infty};\mathcal O)$ being the generator, $b\in H_2(T^{g-1};\Z)$, and generators $1_A\in H_0(\bA(\R P^2);\mathcal O)=\Z_2$ and $1_B\in H_0(T^{g-1};\Z)=\Z$. 
Let $\bar a=a\otimes 1_B$ and $\bar b=1_A\otimes b$.

We apply Proposition~\ref{product mu} with $\mu=\mu_{\bA(\R P^2)\times T^{g-1}}$ and $[N]=[N]_{\mathcal O_N}$ to obtain: 
$$\mu_*(j\times j)_*([N]\times [N])=\mu_*(j_*([N])\times j_*([N])=\mu_*((\bar a+\bar b)\times(\bar a+\bar b))$$
$$=\mu_*(\bar a\times \bar a+\bar a\times \bar b+\bar b\times \bar a+\bar b\times \bar b)=\mu_*(\bar a\times\bar a)+\mu_*(\bar a\times \bar b+\bar b\times\bar a)+\mu_*(\bar b\times\bar b)$$
$$
=\mu^1_*(a\times a)\times 1_B+\mu_*^1(a\times 1_A)\times\mu^2_*(1_B\times b)+\mu_*^1(1_A\times a)\times\mu^2_*(b\times 1_B)+1_A\times\mu^2_*(b\times b)
$$
where $\mu^1=\mu_{\bA(\R P^2)}$ and $\mu^2=\mu_{T^{g-1}}$. By Proposition~\ref{a}, $\mu^1_*(a\times a)\times 1_B=0$.

We recall that a $\Z$-twisted homology class in a space $X$ with a local system $p:E\to X$ is defined by a cycle in $X$ with coefficients in the sections of
$p$ on 
(singular) simplices in $X$. One can assume that the sections are taken in the $\pm 1$-subbundle of the $\Z$-bundle $p$. This implies that every homology class
is represented by a continuous map $f:M\to X$ of a pseudo-manifold that admits a lift $f':M\to E$ with value in the $\pm 1$-subbundle of $p$. 

One can show that the homology class $a\in H_2(\bA(\R P^2);\mathcal O)$ is realized by a map $f:S^2\to\bA(\R P^2)$ that admits a lift to the $\pm 1$-subbundle of $\mathcal O$. The homology class
$1_A\in H_0(\bA(\R P^2);\mathcal O)$ can be realized by the point representing the unit $0\in \bA(\R P^2)$. Then the class $\mu_*^1(a\times 1_A)$ is realized by the 
same map $$f=\mu^1(f,0):S^2\to\bA(\R P^2),$$ i.e., $\mu_*^1(a\times 1_A)=a$, whereas
the class $\mu_*^1(1_A\times a)$ is realized by the composition $I\circ f$. By  Proposition~\ref{inverse}, $\mu_*^1(1_A\times a)=I_*(a)=a$.
Similarly,
$\mu^2_*(b\times 1_B)=b$, and $\mu^2_*(1_B\times b)=I_*(b)=b$.
Therefore, $\mu_*(\bar a\times \bar b+\bar b\times \bar a)=2(a\times b)$ is divisible by 2.
Then by Proposition~\ref{tor}, $\mu_*(\bar a\times \bar b+\bar b\times\bar a)=0$.

Next we show that $\mu_*^2(b\times b)=0$.  In view of Proposition~\ref{tor} it suffices to show that
 $\mu^2_*(b\times b)=0$ mod 2.
We recall that the Pontryagin product defines the structure of an exterior algebra on $H_*(T^{r})=\Lambda[x_1,\dots, x_r]$.
Thus, $b=\sum_{i<j}\lambda_{ij}x_i\wedge x_j$. Then $$b\times b=\sum_{i<j,k<l}\lambda_{ij}\lambda_{kl}(x_i\wedge x_j)\times(x_k\wedge x_l).$$

Let $\{i,j\}\cap\{k,l\}=\emptyset$. By the argument similar to the above using  Propositions~\ref{product mu} and~\ref{inverse} we obtain that 
$$\mu^2_*((x_i\wedge x_j)\times(x_k\wedge x_l)+(x_k\wedge x_l)\times(x_i\wedge x_j))$$ is divisible by 2.

If $|\{i,j,k,l\}|\le 3$, then the problem can be reduced to a 3-torus. Then
$$\mu^2_*((x_i\wedge x_j)\times(x_k\wedge x_l))=0$$ by the dimensional reason.
\end{proof}

\subsection{Inessentiality of the cofiber}

\begin{prop}\label{abelinization}
For a connected CW complex $X$  the fundamental group $G=\pi_1(X\times X/\Delta X)$ is isomorphic to the abelianization
of $\pi_1(X)$ and the induced homomorphism $\pi_1(X\times x_0)\to\pi_1((X\times X)/\Delta X)$ is  surjective.
\end{prop}
\begin{proof}
Let $q:X\times X\to(X\times X)/\Delta X$ be the quotient map. Since $q$ has connected point preimages,
it induces an epimorphism of the fundamental groups. Suppose that $g=q_*(a,b)$, $g\in G$, $(a,b)\in\pi_1(X)\times\pi_1(X)=\pi_1(X\times X)$. Then $$g=q_*(a,b)=q_*(b,b)q_*(b^{-1}a,1)=eq_*(b^{-1}a,1)$$ 
where $e\in G$ is the unit. This proves the second part.

Let $g,h\in G$. By the second part of the proposition, $g=q_*(a,1)$ and $h=q_*(1,b)$. Since $(a,1)(1,b)=(a,b)=(1,b)(a,1)$, we obtain $gh=hg$.
Thus, $G$ is abelian. We show that $\pi_1(X\times x_0)\to\pi_1((X\times X)/\Delta X)$ is the abelianization homomorphism.

Note that the kernel of $q_*$ is the normal closure of the diagonal subgroup $\Delta\pi_1(X)$ in $\pi_1(X)\times\pi_1(X)$.
Thus every element $(x,1)\in K=ker\{\pi_1(X\times x_0)\to\pi_1((X\times X)/\Delta X)\}$ can be presented as the product
$$
(x,1)=(a_1^{y_1},a_1^{z_1})(a_2^{y_2},a_2^{z_2})\cdots(a_n^{y_n},a_n^{z_n})
$$
for $a_i,y_i,z_i\in\pi_1(X)$ where $a^g=gag^{-1}$. This equality implies that $$(a_n^{-1})^{z_n}(a_{n-1}^{-1})^{z_{n-1}}\cdots(a_1^{-1})^{z_1}=1.$$
Then $x=(a_n^{-1})^{z_n}\cdots(a_1^{-1})^{z_1}a_1^{y_1}\cdots a_n^{y_n}$ lies in the kernel of the abelianization map. Therefore, $K\subset [\pi_1(X),\pi_1(X)]$.
\end{proof}

\begin{prop}
For any $g$ the pseudo-manifold $(N_g\times N_g)/\Delta N_g$ is locally orientable and inessential.
\end{prop}
\begin{proof} We use the notation $N=N_g$.
To check the local orientability it suffices to show that $H_4(W,\partial W)=\Z$ for a regular neighborhood of the diagonal 
$\Delta N$ in $N\times N$.
Since $H_4(W)=H_3(W)=0$, the exact sequence of pair implies $H_4(W,\partial W)=H_3(\partial W)$. 
Note that the boundary $\partial W$ is homeomorphic to the total space of the spherical bundle for the tangent bundle on $N$. 
The spectral sequence for this spherical bundle implies that $$H_3(\partial W)=E^2_{2,1}=H_2(N;\underline{H_1(S^1)})$$ where the local system
$\underline{H_1(S^1)}$ is the orientation sheaf on $N$. Thus, we obtain $H_3(\partial W)=\Z$.

Next, we show that the map $\mu\circ(j\times j)$ is homotopic to $f\circ q$ where $q:N\times N\to (N\times N)/\Delta N$
is the quotient map and $f$ is a map classifying the universal covering  for $(N\times N)/\Delta N$. 
Note that for the fundamental groups, $ker(q_*)$ is the normal closure of the diagonal
$\Delta\pi$ in $\pi\times\pi$. Therefore, $(j\times j)_*(ker(q_*))\subset\Delta(Ab(\pi))=ker(\mu_*)$. Hence
there is a homomorphism $\phi:Ab(\pi)\to Ab(\pi)$ such that $\mu_*\circ(j\times j)_*=\phi q_*$:
$$
\begin{CD}
\pi\times\pi @>q_*>> \pi_1(N\times N/\Delta N)\\
@V(j\times j)_*VV @V\phi VV\\
Ab(\pi)\times Ab(\pi) @>\mu_*>> Ab(\pi).\\
\end{CD}
$$
By Proposition~\ref{abelinization}, $\pi_1(N\times N/\Delta N)=Ab(\pi)=\Z^{g-1}\oplus\Z_2$. Since $\phi$ is surjective the homomorphism $\phi$ is an isomorphism.
The homomorphism $\phi$ can be realized by a map $f:(N\times N)/\Delta N\to \bA(N)$.
Since $f$ induces an isomorphism of the fundamental groups $f$ is a classifying map.
Since the maps $\mu\circ(j\times j)$ and $f\circ q$ with the target space $K(Ab(\pi),1)$ induces isomorphisms of the fundamental groups, they are homotopic.

Finally, we note that the fundamental class of $(N\times N)/\Delta N$ is the image of that of $N\times N$.
Then we apply Proposition~\ref{main} and Proposition~\ref{pseudo}.
\end{proof}

\begin{cor}\label{main2}
$\cat((N\times N)/\Delta N)\le 3$.
\end{cor}
\begin{proof}
We apply Proposition~\ref{KR}.
\end{proof}

\begin{thm}
For a non-orientable surface of genus $g$,
$$\cat((N_g\times N_g)/\Delta N_g)=3.$$
\end{thm}
\begin{proof}
We take $x\in H^1(N_g;\Z_2)$ such that $x^2\ne 0$.
Note that $(x\times 1+1\times x)^2=x^2\times 1+1\times x^2$ in $H^*(N_g\times N_g;\Z_2)$.
Then $$(x\times 1+1\times x)^3=(x^2\times 1+1\times x^2)(x\times 1+1\times x)=
x\times x^2+x^2\times x\ne 0.$$
The restriction of $x\times 1+1\times x$ to the diagonal $\Delta N_g\subset N_g\times N_g$ equals $$x\smile 1+1\smile x=2x$$ 
where $\smile$ is the cup product. Since $2x=0$
in $H^*(N_g;\Z_2)$, we obtain $x\times 1+1\times x=q^*(y)$ for some $y\in H^1((N_g\times N_g)/\Delta N_g;\Z_2)$.
Therefore, $y^3\ne 0$ in $H^*((N_g\times N_g)/\Delta N_g;\Z_2)$. By the cup-length estimate (Theorem~\ref{cl}),
$$\cat((N_g\times N_g)/\Delta N_g)\ge 3.$$ This together with Corollary~\ref{main2} implies the required equality.
\end{proof}

\

\section{Higman's group}

Higman's group $H$ has the following properties~\cite{Hi}: $H$ is acyclic and it has finite 2-dimensional Eilenberg-Maclane complex
$K(H,1)$.

\begin{thm}
Let $K=K(H,1)$ where $H$ is Higman's group. Then
$$
2=\cat(C_{\Delta K})< TC(K)=4.
$$
\end{thm}
\begin{proof}
By Proposition~\ref{abelinization}, $\pi_1(C_{\Delta K})=H_1(H)=0$. Then by Theorem~\ref{cl}
$$\cat(C_{\Delta K})\le(\dim C_{\Delta K})/2=2.$$

The equality $TC(K(H,1))=4$ is a computation by Grant, Lupton and Oprea~\cite{GLO}.
\end{proof}

\section{Topological complexity of non-orientable surfaces}

Let $M$ be an orientable surface, $P=\R P^2$ be the projective plane, and let $q:M\vee P\to P$ denote the collapsing $M$ map.
We denote by $\mathcal O'=(q\times q)^*\mathcal O_{P\times P}$   the pull back
of the orientation sheaf on $P\times P$. 

\begin{prop}
The 4-dimensional homology group of $(M\vee P)^2$ with coefficients in $\mathcal O'$ equals $$H_4((M\vee P)^2;\mathcal O')=\Z\oplus\Z\oplus\Z\oplus\Z.$$ 

Moreover, the inclusions of manifolds $\xi_i:W_i\to(M\vee P)^2$, $i=1,\dots,4$ induce isomorphisms 
$H_4(W_i;\xi^*_i\mathcal O')\to H_4((M\vee P)^2;\mathcal O')$ onto the summands where
$W_1=M^2$, $W_2=P^2$, $W_3=M\times P$,
and $W_4=P\times M$. 
\end{prop}
\begin{proof}
We note that from the Mayer-Vietoris exact sequence for the decomposition $(M\vee P)^2=A\cup B$ with $A=M^2\cup P^2$ and $B=(M\times P)\cup (P\times M)$,
$$
\cdots\to H_4(A;\mathcal O'|_A)\oplus H_4(B;\mathcal O'|_B)\stackrel{\psi}\to H_4((M\vee P)^2;\mathcal O')\to
H_3(M\vee M\vee P\vee P;\mathcal O'|_{\dots})
$$
and dimensional reasons it follows that $\psi$ is an isomorphism. Note that the intersections $M^2\cap P^2$ and $(M\times P)\cap(P\times M)$
in $(M\vee P)^2$ are singletons. Hence,
$A=M^2\vee P^2$ and $B=M\times P\vee P\times M$.
Thus, $\psi$ defines the required isomorphism.
\end{proof}
\begin{cor}
$$H_4((M\vee P)^2;\mathcal O')=H_4(M^2)\oplus H_4(P^2;\mathcal O_{P^2})\oplus H_4(M\times P;\mathcal O_{M\times P})\oplus H_4(P\times M;\mathcal O_{P\times M}).$$
\end{cor}
\begin{proof}
The proof is a verification that the restriction of $\mathcal O'$ to each $W_i$, $i=1,2,3,4$, is the orientation sheaf.
\end{proof}

Let the map $f: M\# P\to (M\# P)/S^1=M\vee P$ be the collapsing of the connected sum circle. Note that the composition $q\circ f$ with the above $q$ takes the orientation sheaf $\mathcal O_P$
to the orientation sheaf $\mathcal O_{M\# P}$.
\begin{prop}\label{fundament}
$f_*([M\# P])=[M^2]+[P^2]+[M\times P]+[P\times M]$.
\end{prop}
\begin{proof}
Let $B\subset \xi(W_i)$ be a 4-ball. Then we claim that the homomorphism
$$
H_4((M\vee P)^2;\mathcal O')\to H^4((M\vee P)^2;(M\vee P)^2\setminus\stackrel{\circ}B;\mathcal O')=H_4(B,\partial B)=\Z
$$
generated by the map of pairs and the excision
is the projection of the direct sum $H_4((M\vee P)^2;\mathcal O')=\Z\oplus\Z\oplus\Z\oplus\Z$ onto the $i$th summand.
This follows from the commutative diagram
$$
\begin{CD}
H_4((M\vee P)^2;\mathcal O') @>>>H_4((M\vee P)^2;(M\vee P)^2\setminus\stackrel{\circ}B;\mathcal O')@<=<<H_4(B,\partial B)=\Z\\
@A\xi_iAA  @. @A=AA\\
H_4(W_i;\xi^*\mathcal O') @>=>> H_n(W_i,W_i\setminus \xi_i^{-1}(\stackrel{\circ}B);\xi^*\mathcal O')@<=<<H_4(B,\partial B)=\Z.\\
\end{CD}
$$

The commutative  diagram
$$
\begin{CD}
H_4(M\# P;\mathcal O_{M\# P}) @>=>> H_4(M\# P, M\# P\setminus\stackrel{\circ}B;\mathcal O_{M\# P})@<=<<H_4(B,\partial B)=\Z\\
@Vf\times fVV  @. @ V=VV\\
H_4((M\vee P)^2;\mathcal O') @>>>H_4((M\vee P)^2;(M\vee P)^2\setminus\stackrel{\circ}B;\mathcal O')@<=<<H_4(B,\partial B)=\Z\\
\end{CD}
$$
shows that the projection of the image $f_*([M\# P])$ of the fundamental class onto the $i$th summand, $i=1,2,3,4$, is a fundamental class.
\end{proof}

The proof of the  following proposition is straightforward.
\begin{prop}\label{retraction}
A retraction $r:X\to A$, $A\subset X$, defines a fiberwise retraction  $\bar r:(p^X)^{-1}(A) \to PA$. Moreover, for each $k$ it defines a fiberwise retraction
$\bar r_k:(p^X_k)^{-1}_X(A)\to\Delta_k(A)$ of the fiberwise joins:
$$
\begin{CD}
\Delta_k(X) @<\supset<<(p^X_k)^{-1}_X(A)@>\bar r_k>>\Delta_k(A)\\
@Vp^X_kVV @Vp^X_k|VV @Vp^A_kVV\\
X\times X @<\supset <<A\times A @>=>> A\times A.\\
\end{CD}
$$
\end{prop}

\

We denote  $$g=(1\vee j)^2:(M\vee \R P^2)^2\to (M\vee\R P^{\infty})^2.$$
It is easy to see that the sheaf $\mathcal O'$ on $(M\vee \R P^2)^2$ is the pull back under $g$ of a sheaf $\tilde{\mathcal O}$ on $(M\vee\R P^{\infty})^2$
which comes from the pull back of the
tensor  product  $\mathcal O\hat\otimes\mathcal O$ of the canonical $Z$-bundles
on $\R P^{\infty}$.

\begin{prop}\label{3-cases}
Let $\kappa\in H^4((M\vee \R P^\infty)^2;\mathcal F)$ be the primary obstruction to 
a section of $$\bar p=p^{M\vee \R P^\infty}_3:\Delta_3(M\vee \R P^\infty)\to(M\vee \R P^\infty)^2.$$
Then 

(1) $[M^2]\cap g^*(\kappa)\ne 0$,

(2) $[(\R P^2)^2]\cap g^*(\kappa)=0$, and

(3) $([M\times \R P^2]+[\R P^2\times M])\cap g^*(\kappa)=0$.
\end{prop}
\begin{proof}
(1) Assume that $[M^2]\cap g^*(\kappa)=0$. Then, $g_*([M^2])\cap \kappa=0$.
This means that the map $\bar p$ admits a section over $M^2\subset (M\vee \R P^\infty)^2$. The collapsing $\R P^\infty$ to a point defines
a retraction $r:M\vee \R P^\infty\to M$. By Proposition~\ref{retraction} the retraction $r$ defines a fiberwise retraction of 
$\bar p^{-1}(M^2)$ onto $\Delta_3(M)$.
This implies that $p^{M}_3:\Delta_3(M)\to M^2$ admits a section. Hence, by Theorem~\ref{ganea}, $TC(M)\le 3$. This contradicts to the fact that $TC(M)=4$.

(2) Since $TC(\R P^2)=3$,  by Corollary~\ref{lift} the map $g$ restricted to $(\R P^2)^2$ admits a lift with respect to $\bar p$.
Hence the primary
obstruction  $o'$ 
to such a lift is zero. Note that $o'=(g^*\kappa)|_{(\R P^2)^2}$ is the restriction to $(\R P^2)^2$ of the image of $\kappa$ under $g^*$.
Hence, $$[\R P^2]\cap g^*(\kappa)=[\R P^2]\cap(g^*\kappa)|_{(\R P^2)^2}=0.$$

(3) Let $\sigma:(M\vee \R P^\infty)^2\to (M\vee \R P^\infty)^2$ be the natural involution: $\sigma(x,y)=(y,x)$. 
We may assume that the map $\sigma$ is cellular.
It defines an involution $\bar\sigma$ on the path space $P(M\vee \R P^\infty)$ and 
involutions $\bar\sigma_k$ on the iterated fiberwise joins $\Delta_k(M\vee \R P^\infty)$.

Let $K=(M\vee \R P^\infty)^2$ be a $\sigma$-invariant CW complex structure with an invariant subcomplex 
$Q=(M\times \R P^\infty)\vee (\R P^\infty)\times M)$.
We claim that there is a section $s:K^{(3)}\to \Delta_3(M\vee \R P^\infty)$  which is $\sigma$-equivariant on $Q^{(3)}$. 
First we fix an invariant section at the wedge point $$s(x_0,x_0)=c_{x_0}+0+0+0\in \Delta_3(M\vee \R P^\infty)$$ where $c_{x_0}$ is the constant path at $x_0$. Then we define our section $s$ on
$Q^{(3)}$ by induction on dimension of simplices. 
We note that  $\sigma(e)\ne e$ for all cells in $Q$ except the wedge vertex.
Assume that an equivariant section $s$ is defined on the $i$-skeleton
$Q^{(i)}$, $i<3$. 
Then for all distinct pairs of $i$-cells $e,\sigma(e)$ we do an extension of $s$  to $e$ and define it on $\sigma(e)$ to be $\bar\sigma_3s\sigma$. 
Note that an extension of $s$ to $e$ exists since the fiber of $\bar p$ is 2-connected. Also, in
view of the 2-connectedness
of the fiber the section $s$ on $Q^{(3)}$ can be extended to $K^{(3)}$. 

Thus, we may assume that the restriction of the obstruction cocycle
to $Q$  is symmetric. Hence, for the obstruction cohomology class we obtain $(\sigma^*\kappa)|_Q=\sigma_0^*(\kappa|_Q)=\kappa|_Q$ where
$\sigma_0=\sigma|_Q$.

Let $$q_0:(M\times \R P^\infty)\vee (\R P^\infty\times M)\to M\times \R P^\infty$$ be the projection to the orbit space of the $\sigma$-action, i.e., the folding map. 
Let $\kappa'=\kappa|_{M\times \bA(\R P^2)}$.
Then $\kappa|_Q=q_0^*(\kappa')$. Note that $\tilde{\mathcal O}$ restricted to $(M\times \R P^\infty)\vee (\R P^\infty\times M)$ equals to 
$q_0^*\tilde{\mathcal O}|_{M\times \R P^\infty}$. Hence the homomorphism in homology $(q_0)_*$ is well-defined.
Since $q_0$ induces an epimorphism of the fundamental groups and takes both classes $g_*[M\times \R P^2]$
and $g_*[\R P^2\times M]$ to $g_*[M\times \R P^2]$, we obtain $$(q_0)_*(g_*[M\times \R P^2]+g_*[\R P^2\times M])\cap \kappa)=
2g_*[M\times \R P^2]\cap \kappa'=0.$$
The last equality follows from the fact that $[\R P^2]$ has order 2 in $\R P^\infty$ (see Proposition~\ref{tor}).

Since $q_0$ induces an epimorphism of the fundamental groups, we obtain $$(g_*[M\times \R P^2]+g_*[\R P^2\times M])\cap \kappa=0.$$
Therefore,  $([M\times \R P^2]+[\R P^2\times M])\cap g^*(\kappa)=0$.
\end{proof}
\begin{cor}\label{co}
$$([M^2]+[(\R P^2)^2]+[M\times \R P^2]+[\R P^2\times M])\cap g^*(\kappa)\ne 0.$$
\end{cor}
\begin{cor}\label{co2}
$$g_*([M^2]+[(\R P^2)^2]+[M\times \R P^2]+[\R P^2\times M])\cap \kappa\ne 0.$$
\end{cor}
\begin{proof} 
First we recall that $g_*$ is well-defined
Note that
$g_*([M^2]+[(\R P^2)^2]+[M\times \R P^2]+[\R P^2\times M])\cap g^*(\kappa))=g_*([M^2]+[(\R P^2)^2]+[M\times \R P^2]+[\R P^2\times M])\cap \kappa.$ Since $g_*$ is an isomorphism in dimension 0, we derive the result from Corollary~\ref{co}.
\end{proof}

\begin{thm}
For $g\ge 4$, $TC(N_g)=4$.
\end{thm}
\begin{proof} First we consider the case when $g$ is odd. Then $N_g=M\#\R P^2$ for an orientable surface $M$ of genus $>1$.
Let $f:N_g=M\#\R P^2\to M\vee\R P^2$ be a map that collapses the connected sum circle.
Clearly, $f$ induces an epimorphism of the fundamental groups. Note that
the orientation sheaf $\mathcal O_{N_g}$ is the pull back $f^*q^*\mathcal O_{\R P^2}$ where $q:M\vee\R P^2\to\R P^2$ is the collapsing map.

We show that the map $g\circ(f\times f)=(1\vee j)f\times(1\vee j)f$ does not admit a lift with respect to
 $$\bar p=p^{M\vee\R P^\infty}_3:\Delta_3(M\vee \R P^\infty)\to(M\vee \R P^\infty)^2.$$
Then, by Corollary~\ref{lift} we obtain the inequality $TC(N_g)\ge 4$.

The primary obstruction $o$ to such a lift is the image $(f\times f)^*g^*(\kappa)$ of the primary obstruction to a section. Note that by Proposition~\ref{fundament} and Corollary~\ref{co2},
$$g_*(f\times f)_*([N_g^2]\cap o)=g_*(f\times f)_*([N_g^2])\cap\kappa=g_*([M^2]+[P^2]+[M\times P]+[P\times M])\cap\kappa\ne 0$$
where $P=\R P^2$.
Therefore, $[N_g^2]\cap o\ne 0$. By the Poincare duality (with local coefficients) we obtain that $o\ne 0$.

When $g>4$ is even, $N_g=M\#\R P^2\#\R P^2$ for an orientable surface $M$ of genus $>1$. We consider the map 
$f:N_g\to M\vee\R P^2$ which is the composition of the quotient map $N_g\to M\vee\R P^2\vee\R P^2$ and union of the folding
map $\R P^2\vee\R P^2\to\R P^2$ and the identity map on $M$. For such $f$ the orientation sheaf on $N_g$ can be pushed forward and the above argument works.

The proof of the case $g=4$ differs from the above by the following. In this cases $N_g=T\#K$ where $T$ is a 2-torus and $K$ is  
the Klein bottle $K$. We consider $f:N_g\to T\vee K$ and show that $f\times f$ does not admit a lift with respect to
$$\bar p=p^{T\vee K}_3:\Delta_3(T\vee K)\to(T\vee K)^2.$$
Let $\kappa\in H^4((M\vee \R P^\infty)^2;\mathcal F')$ be the primary obstruction to 
a section of $\bar p$.
If for the Klein bottle $TC(K)=3$ we make
the following modification of Proposition~\ref{3-cases}: 

(1) $[T^2]\cap\kappa=0$, 

(2) $[K^2]\cap\kappa=0$, 

(3)
$([T\times K]+[K\times T])\cap\kappa\ne 0$. 

The first two conditions follows from the facts that $TC(T)=2$ and the assumption $TC(K)<4$.

Here we prove (3). By the argument of Proposition~\ref{3-cases} we obtain that $$([T\times K]+[K\times T])\cap\kappa=[T\times K]\cap 2\kappa'$$
where $\kappa'$ is the restriction of $\kappa$ to $T\times K$. Since $T\times K$ is aspherical, the fibration $\tilde p^{T\times K}$
is fiberwise homotopy equivalent to the universal covering of $T\times K$. Therefore, by Theorem~\ref{Schwarz2}, the primary obstruction to a section of
the 3rd Ganea fibration $\tilde p^{T\times K}_3$ equals $\beta^4$ where $\beta$ is  the Berstein-Schwarz class for $\pi_1(T\times K)$. 
We note that 4-dimensional cohomology group of $T\times K$ with coefficients in  the orientation sheaf equals $\Z$.
Then the universality of $\beta$ (Theorem~\ref{universal})
implies that $\beta^4$ has infinite order. Let the base point of $(T\times K)$ be $(v_0,v_0)$ where $v_0$ is the wedge point in $T\vee K$.
We observe that the fibration $\tilde p^{T\times K}$ embeds into $p^{T\vee K}$ by means of the following map
$\Psi:P_0(T\times K)\to P(T\vee K)$: $\Psi(\phi)=\overline{\phi_T}\phi_K$ where $\phi=(\phi_T,\phi_K)$, $\overline{\phi_T}$ is the reverse path
and $\overline{\phi_T}\phi_K$ is the concatenation. The fibration $P(T\vee K)$ restricted over $T\times K\subset (T\vee K)^2$ admits a fiberwise retraction
onto $\Psi(P_0(T\times K))$ defined as $r(\phi)=(\overline{q_T\circ\phi},q_K\circ\phi)$ where $\phi:[0,1]\to T\vee K$ is a path that starts in $T$
and terminates in $K$, $q_T:T\vee K\to T$ and $q_K:T\vee K\to K$ are the collapsing maps. Then
the 3rd Ganea fibration embeds into $\bar p$ as a fiberwise retract. This retraction defines a split monomorphism
of the coefficients sheafs from the definition of the primary obstructions. The naturality of the obstructions implies that
$\kappa'$ is the image of $\beta^4$ under a monomorphism. Therefore, $\kappa'$ has an infinite order. Thus, $2\kappa'\ne 0$ and by
the Poincare duality we obtain (3). 

Using the above modification of Proposition~\ref{3-cases} we obtain
$$(f\times f)_*([N_g^2]\cap o)=(f\times f)_*([N_g^2])\cap\kappa=([T^2]+[K^2]+[T\times K]+[K\times T])\cap\kappa\ne 0.$$
Therefore, $[N_g]\cap o\ne 0$. By the Poincare duality  with local coefficients we obtain $o\ne 0$.

If $TC(K)=4$, then we show that the square of the map $f\circ q_K:K\# T\to K$ does not have a lift with respect to
$p_3^K:\Delta_3(K)\to K^2$. Indeed, since $(f\circ q_K)_*([K\# T])=[K]$, the primary obstruction to such a lift $\kappa'=((f\circ q_k)^2)^*(\kappa)$ evaluated on the fundamental class  $[K\# T)^2]$ equlas $\kappa$ evaluated on the fundamental class $[K^2]$ where $\kappa$ is the obstruction to a section of $p^K_3$. Since $TC(K)>3$, we have $\kappa\ne 0$.
By the Poicare duality $[K^2]\cap\kappa\ne 0$. Hence, $\kappa'\ne 0$ and the result follows. 
\end{proof}

REMARK 1. Implicitly our proof of the inequality $TC(N_g)\ge 4$ is based on the zero-divisors cup-length estimate as in~\cite{F}.
Indeed, by Schwarz' theorem~\ref{Schwarz2}, $\kappa=\beta^4$ where 
$\beta\in H^1((M\vee\R P^\infty)^2;\mathcal F_0)$ is the primary obstruction for the section of the fibration
$$p^{M\vee\R P^\infty}:P(M\vee\R P^\infty)\to (M\vee\R P^\infty)^2.$$
Therefore, $\beta$ has the restriction to the diagonal equal zero. We have proved that $\alpha^4\ne 0$ for
the element $\alpha=(f\times f)^*g^*(\beta)\in H^1(N_g\times N_g;\mathcal F_0^*)$ which is trivial on the diagonal.
The local coefficient system $\mathcal F_0^*$ as well as a cocycle $a$ 
representing $\alpha$ can be presented explicitly  (in terms of $\pi_1(N_g)$-modules and cross homomorphisms) as it it was done in~\cite{C}.

REMARK 2. The above technique  does not seem to be applicable 
to $N_3=\R P^2\#\R P^2\#\R P^2$.


\begin{thebibliography}{[CLOT]}
\bibliographystyle{amsalpha}



\bibitem[Be]{Be}
I. Berstein,  On the Lusternik-Schnirelmann category of Grassmannians. Math. Proc.
Camb. Philos. Soc. 79  (1976) 129-134.



\bibitem[Bre]{Bre}
G. Bredon, Sheaf Theory. \emph{Graduate Text in Mathematics}, \textbf{170},
Springer, New York Heidelberg Berlin, 1997.

\bibitem[Bro]{Bro}
K. Brown, Cohomology of Groups. \emph{Graduate Texts in Mathematics},
\textbf{87} Springer, New York Heidelberg Berlin, 1994.






\bibitem[CLOT]{CLOT}
O. Cornea; G. Lupton; J. Oprea; D. Tanr\'e,
Lusternik-Schnirelmann Category.  {\em Mathematical Surveys and
Monographs}, \textbf{103}.  American Mathematical Society, Providence,
RI, 2003.

\bibitem[C]{C}
A. E. Costa, Topological complexity of configuration spaces, Ph.D. Thesis, Durham University, 2010.

\bibitem[DT]{DT}A. Dold, R. Thom
{\em Quasifaserungen und unedliche symmetrische Produkte},
Ann. of Math., 67 (1958), pp. 239--281.

\bibitem[Dr1]{Dr1}
A. Dranishnikov, {\em Free abelian topological groups and collapsing maps}. Topology Appl. 159 (2012), no. 9, 
2353--2356.

\bibitem[Dr2]{Dr2}
 A. Dranishnikov, {\em Topological complexity of wedges and covering maps}. Proc. Amer. Math. Soc. 142 (2014), no. 12, 4365-4376.

\bibitem[Dr3]{Dr3}
A. Dranishnikov, {\em The LS-category of the product of lens spaces}, AGT, to appear.


\bibitem[DR]{DR} A. Dranishnikov, Yu. Rudyak,
On the Berstein-Svarc theorem in dimension 2. Math. Proc. Cambridge Philos. Soc. 
{\bf 146} (2009), no. 2, 407-413.

\bibitem[F]{F} M. Farber,
 Invitation to topological robotics. Zurich Lectures in Advanced Mathematics. European Mathematical Society (EMS), Zürich, 2008.

\bibitem[FTY]{FTY} 
M. Farber, S. Tabachnikov, S. Yuzvinsky, {\em Topological robotics: motion 
planning in projective spaces}, Int. Math. Res. Not. 2003, no. 34, 1853-1870.

\bibitem[GV1]{GV1} 
J.M. Garc\'ia Calcines and L. Vandembroucq. {\em Weak sectional category}, Journal
of the London Math. Soc. 82(3) (2010), 621--642.

\bibitem[GV2]{GV2} 
J.M. Garc\'{i}a Calcines and L. Vandembroucq. {\em Topological complexity and the homotopy cofibre of the diagonal map}
 Math. Z. 274 (2013), no. 1-2, 145--165.

\bibitem[G]{G} M. I. Graev, \emph{Free topological groups}, Izv. Akad.
Nauk SSSR Ser. Mat. 12
(1948), 279-324; English transl.,
Amer. Math. Soc. Transl. (1) 8 (1962), 305-364.


\bibitem[GLO]{GLO} 
M. Grant, G. Lupton, J. Oprea,
{\em New lower bounds for the topological complexity of aspherical spaces}, Topology Appl. 189 (2015), 78-91.
 

\bibitem[Gr]{Gr}
M. Gromov, Metric structures for Riemannian and non-Riemannian spaces, Progress
in Math. 152, Birkhauser, Boston (1999).


\bibitem[Hi]{Hi}
G. Higman, {\em A finitely generated infinite simple group} J. London Math. Soc., 26 (1951) pp. 61--64.


\bibitem[KR]{KR} M. Katz, Yu. Rudyak, Lusternik-Schnirelmann category and systolic category of low-dimensional manifolds. Comm. Pure Appl. Math. 59 (2006), no. 10, 1433-1456.

\bibitem[M]{M} A. A. Markov, \emph{On free topological groups}, Izv.
Akad. Nauk SSSR Ser. Mat. 9 (1945), 3-64; English transl., Amer.
Math. Soc. Transl. (1) 8 (1962), 195-272.


\bibitem[Sch]{Sch} A. Schwarz,  The genus of a fibered space. {\em Trudy
Moskov. Mat. Ob\v s\v c} \textbf {10, 11} (1961 and 1962), 217--272, 99--126,
(in {\em Amer. Math. Soc. Transl.} Series 2, vol \textbf{55} (1966)).




\end{thebibliography}
\end{document}